\documentclass[10pt]{amsart}

\usepackage[colorlinks]{hyperref}
\usepackage{color,graphicx,shortvrb}
\usepackage[latin 1]{inputenc}

\usepackage[active]{srcltx} 

\usepackage{enumerate}
\usepackage{amssymb}

\newtheorem{theorem}{Theorem}[section]

\newtheorem{lemma}[theorem]{Lemma}

\newtheorem{proposition}[theorem]{Proposition}

\def\J#1#2#3{ \left\{ #1,#2,#3 \right\} }

\def\NN{{\mathbb{N}}}
\def\11{\textbf{$1$}}

\DeclareGraphicsExtensions{.jpg,.pdf,.png,.eps}

\begin{document}

\title[Tingley's problem through the facial structure]{Tingley's problem through the facial structure of an atomic JBW$^*$-triple}

\author[F.J. Fern\'{a}ndez-Polo]{Francisco J. Fern\'{a}ndez-Polo}
\author[A.M. Peralta]{Antonio M. Peralta}

\address{Departamento de An{\'a}lisis Matem{\'a}tico, Facultad de
Ciencias, Universidad de Granada, 18071 Granada, Spain.}
\email{pacopolo@ugr.es}
\email{aperalta@ugr.es}




\date{}

\begin{abstract} We prove that every surjective isometry between the unit spheres of two atomic JBW$^*$-triples $E$ and $B$ admits a unit extension to a surjective real linear isometry from $E$ into $B$. This result constitutes a new positive answer to Tignley's problem in the Jordan setting.
\end{abstract}

\maketitle
\thispagestyle{empty}

\section{Introduction}

In a recent contribution we establish that every surjective isometry between the unit spheres of two $B(H)$-spaces extends uniquely to a surjective complex linear or conjugate linear surjective isometry between the corresponding spaces (see \cite{FerPe17b}). This result constitutes a positive answer to Tingley's isometric extension problem \cite{Ting1987} in the setting of $B(H)$-spaces and atomic von Neumann algebras. Solutions to Tingley's problem for compact operators, compact C$^*$-algebras and weakly compact JB$^*$-triples have been previously obtained in \cite{PeTan16,FerPe17}. For additional information on the historic background and the state of the art of Tingley's problem the reader is referred to the introduction of \cite{FerPe17b} and to the monograph \cite{YangZhao2014}.\smallskip

Problems in C$^*$-algebras, von Neumann algebras and operator algebras are often considered in the context of Banach Jordan algebras and Jordan triple systems. Such studies widen the scope and often introduce new ideas and techniques not present in the associative case. The class of JB$^*$-triples have a rich interaction with Banach Space Theory. The spaces in this class enjoy a unique geometry which makes more interesting the study of certain geometric problems in a wider setting. This paper is devoted to extend the recent results in \cite{FerPe17b} to the context of atomic JBW$^*$-triples (i.e. $\ell_{\infty}$-sums of Cartan factors).\smallskip

We recall that a \emph{JB$^*$-triple} is a complex Banach space $E$ which can be
equipped with a continuous triple product $\J ... :
E\times E\times E \to E,$ which is symmetric and linear in the
first and third variables, conjugate linear in the second variable
and satisfies the following axioms
\begin{enumerate}[{\rm (a)}] \item $L(a,b) L(x,y) = L(x,y) L(a,b) + L(L(a,b)x,y)
 - L(x,L(b,a)y),$
where $L(a,b)$ is the operator on $E$ given by$L(a,b) x = \J abx;$
\item $L(a,a)$ is an hermitian operator with non-negative
spectrum; \item $\|L(a,a)\| = \|a\|^2$.\end{enumerate}

Examples of JB$^*$-triples include C$^*$-algebras with respect to the triple product defined by
product\hyphenation{product}\begin{equation}\label{eq product operators} \J xyz =\frac12 (x y^* z +z y^*
x),
\end{equation} and JB$^*$-algebras under the triple
product $\J xyz = (x\circ y^*) \circ z + (z\circ y^*)\circ x -
(x\circ z)\circ y^*.$ The so-called \emph{ternary rings of operators} (TRO's) studied, for example, in \cite{NeRu} are also examples of JB$^*$-triples.\smallskip

A subtriple $\mathcal{I}$ of a JB$^*$-triple  $E$ is said to be an \emph{ideal} or a \emph{triple ideal} of $E$ if $\{ E,E,\mathcal{I}\} + \{ E,\mathcal{I},E\}\subseteq \mathcal{I}$.\smallskip

A JBW$^*$-triple is a JB$^*$-triple which is also a dual Banach space (with a unique isometric predual \cite{BarTi86}). It is known
that the second dual of a JB$^*$-triple is a JBW$^*$-triple (compare \cite{Di86}). An extension of Sakai's theorem assures that the triple product of every JBW$^*$-triple is separately weak$^*$-continuous (c.f. \cite{BarTi86} or \cite{Horn87}).\smallskip

Another illustrative examples of JBW$^*$-triples are given by the so-called Cartan factors. A complex Banach space is a \emph{Cartan factor of type 1} is it coincides with the complex Banach space $L(H, K)$, of all bounded linear operators between two complex Hilbert spaces, $H$ and $K$, whose triple product is given by \eqref{eq product operators}.\smallskip

Given a conjugation, $j$, on a complex Hilbert space, $H$, we can define a linear involution on $L(H)$ defined by $x \mapsto x^{t}:=j x^* j$. A \emph{type 2 Cartan factor} is a subtriple of $L(H)$ formed by the skew-symmetric operators for the involution $t$; similarly, a \emph{type 3 Cartan factor} is formed by the $t$-symmetric operators.  A Banach space $X$ is called a \emph{Cartan factor of type 4} or \emph{spin} if $X$ admits a complete inner product $(.|.)$ and a conjugation $x\mapsto \overline{x},$ for which the norm of $X$ is given by $$ \|x\|^2 = (x|x) + \sqrt{(x|x)^2 -|
(x|\overline{x}) |^2}.$$ \emph{Cartan factors of types 5 and 6} (also called \emph{exceptional} Cartan factors) are all finite dimensional. An atomic JBW$^*$-triple is a JBW$^*$-triple which can be represented as an $\ell_{\infty}$-sum of Cartan factors. We refer to \cite{Chu2012} for additional details.\smallskip

Let $E$ and $B$ be atomic JBW$^*$-triples. In our main result we prove that every surjective isometry $f: S(E)\to S(B)$ admits a unique extension to a surjective real linear isometry $T : E \to B$ (see Theorem \ref{t Tingley on atomic JBW-triples}). This theorem extends the main conclusion in \cite{FerPe17b} to the setting of atomic JBW$^*$-triples. In \cite{FerPe17b}, we strived for arguments essentially based on standard techniques of C$^*$-algebras, Geometry and Functional Analysis. The proofs here are extended to the setting of JBW$^*$-triples with new and independent techniques which could be also applied to C$^*$-algebras.\smallskip

As in recent contributions studying Tingley's problem on C$^*$-algebras and von Neumann algebras, our arguments are based on an useful result due to L. Cheng, Y. Dong and R. Tanaka, which asserts that every surjective isometry between the unit spheres of two Banach spaces $X$ and $Y$ maps maximal proper {\rm(}norm closed{\rm)} face of the unit ball of $X$ to maximal proper {\rm(}norm closed{\rm)} face of the unit ball of $Y$ {\rm(}see \cite[Lemma 5.1]{ChengDong}, \cite[Lemma 3.5]{Tan2014} and \cite[Lemma 3.3]{Tan2016}{\rm)}.\smallskip

Throughout the paper, given a Banach space $X$ the symbols $\mathcal{B}_X$ and $S(X)$ will stand for the closed unit ball and the unite sphere of $X$, respectively.\smallskip

By a result of C.M. Edwards, C. Hoskin and the authors of this note (see \cite{EdFerHosPe2010}) we know that for each non-empty norm closed face $F$ of the unit ball $\mathcal{B}_E$ in a JB$^*$-triple $E$ there exists a unique compact tripotent $u$ in $E^{**}$ such that
$$ F = F_u=(u + E_0^{**}(u)) \cap \mathcal{B}_E ,$$ where $E_0^{**}(u)$ is the Peirce-zero space associated with $u$  in $E^{**}$.\smallskip

An appropriate application of Kadison's transitivity theorem for JB$^*$-triples (\cite[Theorems 3.3 and 3.4]{BuFerMarPe}) proves that every maximal proper norm closed face of $\mathcal{B}_E$ is of the form \begin{equation}\label{eq max closed faces} F_e= (e + E_0^{**}(e)) \cap \mathcal{B}_E,
\end{equation} for a unique minimal tripotent $e$ in $E^{**}$. However this minimal tripotent $e$ need not be in $E$.\smallskip

We recall that every atomic JBW$^*$-triple coincides with the weak$^*$-closure of the linear span of all its minimal tripotents (see \cite{FriRu86}).\smallskip

If $B$ is another atomic JBW$^*$-triple and $f : S(E) \to S(B)$ is a surjective isometry, for each minimal tripotent $e\in E$, there exists a minimal tripotent $u$ in $B^{**}$ satisfying $$f(F_e) = f((e + E_0(e)) \cap \mathcal{B}_E) = (u + B_0^{**}(u)) \cap \mathcal{B}_B.$$ However, as in the case of Tingley's theorem for surjective isometries between the unit spheres of $B(H)$-spaces (see \cite{FerPe17b}), when dealing with maximal proper faces in $\mathcal{B}_{B}$, minimal tripotents in $B^{**}$ need not be in $B$. To avoid this difficulty, in this paper we shall prove that, under the above conditions, the minimal tripotent $u$ belongs to $B$ (see Theorem \ref{t surjective isometries between the spheres preserve minimal tripotents}). The proofs in this note are based on geometric arguments combined with Functional Analysis techniques.

\section{Tingley's problem for atomic JBW$^*$-triples}

An element $u$ in a JB$^*$-triple $E$ is called tripotent if $\{u,u,u\} = u$. Each tripotent $u$ in $E$ induces a \emph{Peirce decomposition} of $E$ in the form $$E= E_{2} (u) \oplus E_{1} (u) \oplus E_0 (u),$$ where for each $i\in \{0,1,2\}$ the space $E_i (u)$ is precisely the
$\frac{i}{2}$ eigenspace of $L(u,u)$. Peirce arithmetic assures that $\J {E_{i}(u)}{E_{j} (u)}{E_{k} (u)}$ is contained in $E_{i-j+k} (u)$ if $i-j+k \in \{ 0,1,2\},$ and is zero otherwise. In addition, $$\J {E_{2} (u)}{E_{0}(u)}{E} = \J {E_{0} (u)}{E_{2}(u)}{E} =0.$$
The corresponding \emph{Peirce projections}, $P_{i} (u) : E\to
E_{i} (u)$, $(i=0,1,2)$ are contractive and satisfy $$P_2 (u) = L(u,u)
(2 L(u,u) -I), \ P_1 (u) = 4 L(u,u) (I-L(u,u)),$$  and $P_0 (u) = (I-L(u,u))
(I-2 L(u,u))$ (compare \cite{FriRu85}).\medskip

A non-zero tripotent $e$ in a JB$^*$-triple $E$ is called \emph{minimal} (respectively, \emph{complete} or \emph{maximal}) if $E_2 (e) = \mathbb{C} e$ (respectively, $E_0(e) =0$).\smallskip

Let $x$ be an element in a JB$^*$-triple $E$. We will denote by $E_x$ the JB$^*$-subtriple of $E$ generated by $x$, that is, the closed subspace generated by all odd powers of the form $x^{[1]} := x$, $x^{[3]} := \J xxx$, and $x^{[2n+1]} := \J xx{x^{[2n-1]}},$ $(n\in \NN)$. It is known that $E_x$ is JB$^*$-triple isomorphic (and hence isometric) to a commutative C$^*$-algebra in which $x$ is a positive generator (cf. \cite[Corollary 1.15]{Ka83}). Actually, there exist a (unique) subset Sp$(x)\subset [0,\|x\|]$ (called the \emph{triple spectrum} of $x$) such that $\|x\|\in $Sp$(x) \cup \{0\}$ and the latter is compact, and a triple isomorphism $\Psi : E_x\to C_0(\hbox{Sp} (x))$ mapping $x$ into the function $t\mapsto t$ (compare \cite{Ka96}).\smallskip

Suppose $x$ is a norm-one element. The sequence $(x^{[2n -1]})$ converges in the weak$^*$-topology of $E^{**}$ to a tripotent (called the \hyphenation{support}\emph{support} \emph{tripotent} of $x$) $u(x)$ in $E^{**}$ (see \cite[Lemma 3.3]{EdRutt88} or \cite[page 130]{EdFerHosPe2010}).\smallskip

We recall at this stage a result taken from \cite{FerPe17}.

\begin{proposition}\label{p l 3.4 new}{\rm\cite[Proposition 2.2]{FerPe17}} Let $e$ and $x$ be norm-one elements in a JB$^*$-triple $E$. Suppose that $e$ is a minimal tripotent and $\|e-x\| = 2$. Then the identity $x= -e + P_0(e) (x)$ holds. $\hfill\Box$
\end{proposition}

Following the standard notation, for each ultrafilter $\mathcal{U}$ on an index set $I$, and each family $(X_i)_{i\in I}$ of Banach spaces, we denote by $(X_i)_{\mathcal{U}}$
the corresponding ultraproduct of the $X_i,$ and if $X_i=X$ for all $i$, we write $(X)_{\mathcal{U}}$ for the ultrapower of $X$. As usually, elements in $(X_i)_{\mathcal{U}}$ will be denoted in the form $\widetilde{x}=[x_i]_{\mathcal{U}}$, where $(x_i)$ is called
a representing family or a representative of $\widetilde{x}$, and $\left\|{\widetilde{x}}\right\|=\lim_{\mathcal{U}}\left\|{x_i}\right\|$ independently of the representative. The basic facts and definitions concerning ultraproducts can be found in \cite{Hein80}.\smallskip

In our concrete setting, we can reduce our attention to a family $(E_i )_{i\in I}$ of JB$^*$-triples. It is known that JB$^*$-triples are stable under $\ell_{\infty}$-sums (see \cite[p. 523]{Ka83}), thus the Banach space $\ell_{\infty} (E_i)$ is a JB$^*$-triple with pointwise operations. The closed subtriple ${c}_0 \left(E_i\right)$ is a triple ideal of $\ell_{\infty} (E_i)$, and hence $(E_i)_{\mathcal{U}}= \ell_{\infty} (E_i)/{c}_0 \left(E_i\right)$ is a JB$^*$-triple (see \cite{Ka83}).\smallskip

Our next goal is a quantitative version of the previous Proposition \ref{p l 3.4 new}.

\begin{proposition}\label{p l 3.4 new quatitative} Let $e$ be a minimal tripotent in a JB$^*$-triple $E$. Then for each $\varepsilon >0$ there exists $\delta>0$ satisfying the following property: for each $x\in S(E)$ with $\|e-x \| > 2 -\delta$ we have $\| P_2 (e) (e -x) \|> 2-\varepsilon.$
\end{proposition}

\begin{proof} Arguing by reduction to the absurd, we assume the existence of $\varepsilon >0$ such that for each natural $n$ there exists $x_n\in S(E)$ with $\|e-x_n\| > 2-\frac1n$ and $\| P_2 (e) (e -x_n) \|\leq 2-\varepsilon,$ for every natural $n$.\smallskip

Let $\mathcal{U}$ be a free ultrafilter over $\mathbb{N}$. We consider the JB$^*$-triple $E_{\mathcal{U}}$. Clearly the element $[e]_{\mathcal{U}}$ is a tripotent in $E_{\mathcal{U}}$. We claim that $[e]_{\mathcal{U}}$ is a minimal tripotent. Indeed, suppose  $\widetilde{x}=[x_n]_{\mathcal{U}}\in (E_{\mathcal{U}})_2 ([e]_{\mathcal{U}})$, then $$[x_n]_{\mathcal{U}} = P_2 ([e]_{\mathcal{U}}) [x_n]_{\mathcal{U}} = [P_2 (e) (x_n) ]_{\mathcal{U}} = [\lambda_n e ]_{\mathcal{U}},$$ where $(\lambda_n)$ is a bounded family in $\mathbb{C}$. By compactness arguments the limit $\lim_{\mathcal{U}} (\lambda_n) = \lambda_0$ exists in $\mathbb{C}$. By the triangular inequality we have $$\left\| [x_n]_{\mathcal{U}} -\lambda_0 [e]_{\mathcal{U}} \right\|\leq \left\| [x_n]_{\mathcal{U}} - [ \lambda_n e]_{\mathcal{U}} \right\| + \left\| [ \lambda_n e]_{\mathcal{U}} -\lambda_0 [e]_{\mathcal{U}} \right\| =  \lim_{\mathcal{U}}|\lambda_n -\lambda_0|,$$ and hence $[x_n]_{\mathcal{U}} = \lambda_0 [e]_{\mathcal{U}} \in \mathbb{C} [e]_{\mathcal{U}}$.\smallskip

Finally, since $\|[x_n]_{\mathcal{U}}\| = 1$, and $2\geq \| [e]_{\mathcal{U}} - [x_n]_{\mathcal{U}}\| = \lim_{\mathcal{U}} \| e-x_n\| \geq \lim_{\mathcal{U}}  2-\frac1n =2.$ Applying Proposition \ref{p l 3.4 new} we get $$[x_n]_{\mathcal{U}} = [-e]_{\mathcal{U}}+ P_0 ([e]_{\mathcal{U}}) [x_n]_{\mathcal{U}} = [-e+P_0 (e)(x_n)]_{\mathcal{U}} ,$$ and thus $$[P_2 (e) (e -x_n)]_{\mathcal{U}}=P_2 ([e]_{\mathcal{U}}) ([e]_{\mathcal{U}} - [x_i]_{\mathcal{U}}) = 2[e]_{\mathcal{U}},$$ and
$$\lim_{\mathcal{U}} \|P_2 (e) (e -x_n) \| = \lim_{\mathcal{U}} \| 2 e \| =2,$$ which contradicts $\| P_2 (e) (e -x_n) \|\leq 2-\varepsilon,$ for every natural $n$.
\end{proof}

A common tool applied in all recent studies on Tingley's problem for non-commutative structures is based on an appropriate description of the facial structures of the closed unit balls of the involved spaces. The justification is essentially due to the following result established by L. Cheng, Y. Dong and R. Tanaka.

\begin{proposition}\label{p preserves maximal convex subsets}{\rm(}\cite[Lemma 5.1]{ChengDong}, \cite[Lemma 3.5]{Tan2014} and \cite[Lemma 3.3]{Tan2016}{\rm)} Let $X$, $Y$ be Banach spaces, and let $T : S(X) \to S(Y )$ be a surjective isometry. Then $C$ is a
maximal convex subset of $S(X)$ if and only if $T(C)$ is that of $S(Y)$. Then $C$ is a maximal proper {\rm(}norm closed{\rm)} face of $\mathcal{B}_X$ if and only if $f(C)$ is a maximal proper {\rm(}norm closed{\rm)} face of $\mathcal{B}_Y$.$\hfill\Box$
\end{proposition}

Accordingly to the notation in \cite{EdRu96} and \cite{FerPe06}, a tripotent $e$ in the second dual, $E^{**}$, of a JB$^*$-triple $E$ is said to be \emph{compact-$G_{\delta}$} if there exists a norm one element $a$ in $E$ such that $e$ is the support tripotent of $a$. A tripotent $e$ in $E^{**}$ is called \emph{compact} if $e=0$ or it is the infimum of a decreasing net of compact-$G_{\delta}$ tripotents in $E^{**}$.\smallskip

The norm closed faces of the closed unit ball of a C$^*$-algebra were determined by C.A. Akemann and G.K. Pedersen in \cite{AkPed92}. Their characterization played a decisive role in the arguments presented in \cite{Tan2016,Tan2016-2,Tan2016preprint,PeTan16} and \cite{FerPe17b}. The result of Akemann and Pedersen was extended to the strictly wider setting of JB$^*$-triples by C.M. Edwards, C. Hoskin and the authors of this note in \cite{EdFerHosPe2010}. For later purposes we recall a theorem borrowed from the just quoted paper.

\begin{theorem}\label{thm norm closed faces}\cite{EdFerHosPe2010} Let $E$ be a JB$^*$-triple, and let $F$ be a non-empty norm closed face of
the unit ball $\mathcal{B}_E$ in $E$. Then, there exists a unique compact tripotent
$u$ in $E^{**}$ such that
\[
F = F_u=(u + E_0^{**}(u)) \cap \mathcal{B}_E ,
\]
where $E_0^{**}(u)$ is the Peirce-zero space associated with $u$  in $E^{**}$. Furthermore, the mapping $u \mapsto F_u$ is an anti-order isomorphism from the lattice $\tilde{\mathcal{U}}_c( E^{**})$ of all compact tripotents in $E^{**}$ onto the complete lattice of norm closed faces of $\mathcal{B}_E$.$\hfill\Box$
\end{theorem}

Let $e$ be a tripotent in a JB$^*$-triple $E$. We shall say that $e$ is a \emph{finite-rank tripotent} if $e$ can be written as a finite sum of mutually orthogonal minimal tripotents in $E$. An appropriate extension of Kadison's transitivity theorem for JB$^*$-triples, established in \cite[Theorems 3.3 and 3.4]{BuFerMarPe}, proves that each finite-rank tripotent $e$ in the bidual $E^{**}$ of $E$ is compact, and it is further known that \begin{equation}\label{eq Peirce and L(e,e) remains on E} P_2 (e) (E^{**}) = P_2 (e) (E) = \mathbb{C} e, \hbox{ and } P_1 (e) (E^{**}) = P_1 (e) (E),
\end{equation} where $P_j (e)$ stands for the $j-$th Peirce projection associated with $e$ in  $E^{**}$. Accordingly to these comments and the previous Theorem \ref{thm norm closed faces}, every maximal proper norm closed face of $\mathcal{B}_E$ is of the form \begin{equation}\label{eq max closed faces} F_e= (e + E_0^{**}(e)) \cap \mathcal{B}_E,
\end{equation} for a unique minimal tripotent $e$ in $E^{**}$. However this minimal tripotent $e$ need not be in $E$.\smallskip

We recall that elements $a,b$ in a JB$^*$-triple $E$ are said to be \emph{orthogonal} (written $a\perp b$) if $L(a,b) =0$ (see \cite[Lemma\ 1]{BurFerGarMarPe} for additional details). It follows from Peirce arithmetic that, for each tripotent $e\in E$, $E_2(e)\perp E_0(e)$. The relation ``being orthogonal'' can be applied to define a partial order in the set of tripotents in $E$ given by $u\leq e$ if $e-u$ is a tripotent orthogonal to $e$ (see, for example, \cite{FriRu85} or \cite{Horn87}). It is known that in a JBW$^*$-triple $M$ a tripotent $e\in M$ is minimal if and only if it is minimal for the order $\leq$.\smallskip

We are now in position to extend \cite[Theorem 2.3]{FerPe17b}.

\begin{theorem}\label{t surjective isometries map minimal tripotents into points of strong subdiff} Let $E$ and $B$ be JB$^*$-algebras, and suppose that $f: S(E) \to S(B)$ is a surjective isometry. Let $e$ be a minimal tripotent in $E$. Then $1$ is isolated in the triple spectrum of $f(e)$.
\end{theorem}

\begin{proof} Arguing by contradiction, we assume that 1 is not isolated in $\hbox{Sp} (f(e))$. We shall identify $B_{f(e)}$ with $C_0 (\hbox{Sp} (f(e)))$.\smallskip

By Proposition \ref{p preserves maximal convex subsets}, Theorem \ref{thm norm closed faces}, and \eqref{eq max closed faces}, there exists a minimal tripotent $u\in B^{**}$ such that \begin{equation}\label{eq faces thm 1} f(F_e) = f((e + E_0(e)) \cap \mathcal{B}_E) ) = f((e + E_0^{**}(e)) \cap \mathcal{B}_E) ) = F_u= (u + B_0^{**}(e)) \cap \mathcal{B}_B.
\end{equation}

For each natural $n,$ we define $\hat{x}_n, \hat{y}_n$ the elements in $B_{f(e)}$ given by:
$$\hat{x}_n(t):=\left\{%
\begin{array}{ll}
    \frac{t}{t_0}, & \hbox{if $0\leq t\leq 1-\frac1n$} \\
    \hbox{affine}, & \hbox{if $1-\frac1n\leq t\leq 1-\frac{1}{2n}$} \\
    0, & \hbox{if $1-\frac{1}{2n}\leq t\leq 1$} \\
\end{array}%
\right. \  ; \hat{y}_n(t):=\left\{%
\begin{array}{ll}
    0, & \hbox{if $0\leq t\leq 1-\frac{1}{2n}$} \\
    \hbox{affine}, & \hbox{if $1-\frac{1}{2n}\leq t\leq 1$} \\
    1, & \hbox{if $ t=1$.} \\
\end{array}%
\right.$$

Clearly $\hat{x}_n,\hat{y}_n\in S(B)$ and $\hat{x}_n\perp\hat{y}_n$. We claim that $\hat{y}_n\in (u+B^{**}_0 (u))\cap \mathcal{B}_B = F_{u}$. Indeed, the support tripotent of $f(e)$ is bigger than or equal to $u$ in $B^{**}$, that is, $f(e) = u + P_0(u) (f(e))$ in $B^{**}$ (see \eqref{eq faces thm 1}). Thus, the support tripotent of $\hat{y}_n $ also is bigger than or equal to $u$ in $B^{**}$, that is, $\hat{y}_n = u + P_0(u) (\hat{y}_n) \in F_u$.\smallskip

Therefore, denoting by $x_n=f^{-1}(-\hat{x}_n)\in S(E)$ and $y_n=f^{-1}(\hat{y}_n)\in F_e=(e+E^{**}_0(e))\cap \mathcal{B}_E =(e+E_0(e))\cap \mathcal{B}_E$, we have $$1=\|\hat{x}_n+\hat{y}_n\|=\|\hat{y}_n- (-\hat{x}_n)\|=\|y_n-x_n\|,$$  $$2-\frac1n=\|f(e)+\hat{x}_n\|=\|e-x_n\|,$$ and $y_n = e + P_0 (e) (y_n),$ for every natural $n$.\smallskip

Applying Proposition \ref{p l 3.4 new quatitative} we can find a natural $n_0$ such that $$\|P_2 (e) (e- x_{n_0})\|>\frac32 >1.$$

Finally, the inequalities $$1\geq \|P_2(e) (y_{n_0} -x_{n_0} )\|=\|P_2(e)(e+P_0(e) (y_{n_0})-x_{n_0})\|=\|P_2(e)(e-x_{n_0})\|>\frac32,$$ give the desired contradiction.\end{proof}

In our list of ingredients to extend the results in \cite{FerPe17b}, the next goal is a generalization of \cite[Lemma 2.4]{FerPe17b}.

\begin{lemma}\label{l minimal in E** not in E}
Let $E$ be a JB$^*$-triple. Then every minimal tripotent $u$ in $E^{**}\backslash E$ is orthogonal to all minimal tripotents in $E$.
\end{lemma}

\begin{proof} Suppose, contrary to what we want to prove, that there exist $e\in E$ such that $u$ is not orthogonal to $e$. The atomic part of $E^{**}$ is precisely the JBW$^*$-subtriple of $E^{**}$ generated by all minimal tripotents in $E^{**}$ (see \cite[Theorem 2]{FriRu85}) and coincides with an $\ell_{\infty}$-sum of a certain family of Cartan factors (compare \cite[Proposition 2]{FriRu86}).\smallskip

We are in position to apply \cite[Section 3, page 16, $(\checkmark.1)$ and $(\checkmark.2)$]{FerPe17} to assure that one of the following statements holds:
\begin{enumerate}[$(a)$] \item There exists a quadrangle $(v_1,v_2,v_3,v_4)$ (that is, $v_{1}\bot v_{3}$, $v_{2}\bot v_{4}$, $v_{j}\in E^{**}_1 (v_{j+1})$, $v_{j+1}\in E^{**}_1 (v_{j})$ for every $j\in\{1,2,3\}$,  $v_{1}\in E^{**}_1 (v_{4})$, $v_{4}\in E^{**}_1 (v_{1})$ and $v_{4}=2 \J
{v_{1}}{v_{2}}{v_{3}}$) of minimal tripotents in a Cartan factor contained in the atomic part of $E^{**}$ and complex numbers $\alpha, \beta, \gamma, \delta$ such that $e=v_1$, $|\alpha|^2 + |\beta|^2+ |\gamma|^2 +|\delta|^2=1, $ $\alpha \delta=\beta \gamma $ and $u=\alpha e+\beta v_2+\delta v_3+\gamma v_4$;
\item There exists a trangle $(e,v,\tilde{e})$ (i.e. $e\bot \tilde e$, $e, \tilde e \in E^{**}_{2} (v)$, $v\in E^{**}_{1} (e)$,  $v\in E^{**}_{1} (\tilde e)$ and $e = \{v, \tilde e,v\}$) with $v$ a rank-2 tripotent and $\tilde{e}$ a minimal tripotent in a Cartan factor contained in the atomic part of $E^{**}$ and complex numbers $\alpha, \beta, \delta$ such that $|\alpha|^2 + 2 |\beta|^2+|\delta|^2=1, $ $\alpha \delta=\beta^2 $ and $u=\alpha e+\beta v+\delta \tilde{e}$.
\end{enumerate}

By assumptions $e\not\perp  u$, and thus $\delta\neq 1$. The minimality of $e$ in $E$ can be combined with Kadison's trasitivity theorem (compare \cite[Theorems 3.3 and 3.4]{BuFerMarPe} and \eqref{eq Peirce and L(e,e) remains on E}) to deduce that $P_1 (e) (E^{**}) \subseteq E$.\smallskip

In the case $(a)$, we have $v_2,v_4\in E^{**}_1 (e) = P_1 (e) (E^{**}) = P_1 (e) (E) \subseteq E$ because $e\in E$. Another application of Kadison's trasitivity theorem \eqref{eq Peirce and L(e,e) remains on E} shows that $v_3 \in E^{**}_1 (v_2) = P_1 (v_2) (E^{**}) = P_1 (v_2) (E) \subseteq E$ because $v_2\in E$. Therefore, $v_1,v_2,v_3$ and $v_4$ belong to $E$ and hence $u =\alpha e+\beta v_2+\delta v_3+\gamma v_4 \in E$, which is impossible.\smallskip

In case $(b)$ we have $v\in E^{**}_1 (e) = P_1 (e) (E^{**}) = P_1 (e) (E) \subseteq E$ because $e\in E$. Since $v$ is a rank-2 tripotent, Kadison's trasitivity theorem \eqref{eq Peirce and L(e,e) remains on E} proves that $\tilde e \in E^{**}_1 (v) = P_1 (v) (E^{**}) = P_1 (v) (E) \subseteq E$ because $v\in E$. This shows that $u=\alpha e+\beta v+\delta \tilde{e}\in E$ leading us to a contradiction.
\end{proof}

Our next result is the real core of the study of surjective isometries between the unit spheres of two atomic JBW$^*$-triples.

\begin{theorem}\label{t surjective isometries between the spheres preserve minimal tripotents}
Let $E$ and $B$ be atomic JBW$^*$-triples, and suppose that $f: S(E) \to S(B)$ is a surjective isometry. Then, for each minimal tripotent $e$ in $E$ there exists a unique minimal tripotent $u$ in $B$ such that $f(e)= u$. Moreover, there exists a real linear surjective isometry $T_e : E_0 (e) \to B_0(u)$ such that $$f(e +x )= u + T_e (x),$$ for every $x\in \mathcal{B}_{E_0 (e)}$, and the restriction of $f$ to the maximal norm-closed face $F_e = e+ \mathcal{B}_{E_0(e)}$ is an affine function.
\end{theorem}

\begin{proof} Combining Proposition \ref{p preserves maximal convex subsets}, Theorem \ref{thm norm closed faces}, and \eqref{eq max closed faces}, we find a minimal tripotent $u\in B^{**}$ such that \begin{equation}\label{eq faces thm 2} f(F_e) = f((e + E_0(e)) \cap \mathcal{B}_E) )= f((e + E_0^{**}(e)) \cap \mathcal{B}_E) ) = F_u= (u + B_0^{**}(e)) \cap \mathcal{B}_B.
\end{equation}

We claim that $u \in B$. If on the contrary $u\in B^{**}\backslash B$, Lemma \ref{l minimal in E** not in E} implies that $u\perp v$ for every minimal tripotent $v\in B$. Theorem \ref{t surjective isometries map minimal tripotents into points of strong subdiff} proves that $1$ is an isolated point in the triple spectrum of $f(e)$.\smallskip

As before, we shall identify $B_{f(e)}$ with $C_0 (\hbox{Sp} (f(e)))$. Since $1$ is isolated in $\hbox{Sp} (f(e))$, the element $\hat{w}=\chi_{\{1\}} (f(e))$ is a tripotent in $B$ (actually $\hat{w}$ is the support tripotent of $f(e)$). Having in mind that $f(e) \in F_u$, we have $f(e) = u + P_0(u) (f(e))$ and hence $u \leq \hat{w}$.\smallskip

Since $B$ is atomic, we can find a minimal tripotent $\hat{w}_0$ in $B$ satisfying $\hat{w}_0\leq \hat{w}$. By Lemma \ref{l minimal in E** not in E} and the assumptions we have $u\perp \hat{w}_0$, and thus $u\leq \hat{w}-\hat{w}_0$. Clearly, $$2=\|f(e)+\hat{w_0}\|=\|f(e)-(-\hat{w}_0)\| =\|e-f^{-1}(-\hat{w}_0)\|$$ and, by Proposition \ref{p l 3.4 new} \cite[Proposition 2.2]{FerPe17} or Proposition \ref{p l 3.4 new quatitative}, we have  $$w_0=f^{-1}(-\hat{w}_0)= -e +P_0(e) (w_0).$$ Having in mind that $f(e)-\hat{w}_0 \in (u + B^{**}_0 (u))\cap \mathcal{B}_B$, we deduce that the element $z=f^{-1}(f(e)-\hat{w}_0)$ belongs to $(e+E^{**}_0(e))\cap \mathcal{B}_E$, and thus $z= e + P_0(e) (z)$, and therefore  $$1=\|f(e)\|=\|f(e)-\hat{w}_0+\hat{w}_0\|=\|(f(e)-\hat{w}_0)-(-\hat{w}_0)\|=\|z-w_0\|$$ $$=\|e+ P_0(e)(z)-(-e+P_0(e)(w_0))\|= \| 2e + P_0(e)(z-w_0)\|= 2,$$ which is impossible.\smallskip

We have therefore shown that $u\in B$. We shall now mimic the arguments in the proof of \cite[Proposition 3.1]{PeTan16}, details are enclosed for completeness reasons. Since $$f\left( e + \mathcal{B}_{E_0(e)} \right)=f((e + E_0(e)) \cap \mathcal{B}_E ) = F_u = (u + B_0(u)) \cap \mathcal{B}_B = u + \mathcal{B}_{B_0(u)},$$ denoting by $\mathcal{T}_{x_0}$ the translation with respect to $x_0$ (i.e. $\mathcal{T}_{x_0} (x) = x+x_0$), the mapping $f_{e} = \mathcal{T}_{u}^{-1}|_{F_u} \circ f|_{F_e} \circ \mathcal{T}_{e}|_{\mathcal{B}_{E_0(e)}}$ is a surjective isometry from $\mathcal{B}_{E_0(e)}$ onto $\mathcal{B}_{B_0(u)}$. Mankiewicz's theorem (see \cite{Mank1972}) implies the existence of a surjective real linear isometry $T_{e} : E_0(e) \to  B_0(u)$ such that $f_{e} = T_{e}|_{S(E_0(e))}$ and hence $$ f(e + x) = u + T_e (x), \hbox{ for all $x$ in } \mathcal{B}_{E_0(e)}.$$ In particular $f(e) = u$. Now, since $$f|_{F_e}= \mathcal{T}_{u}|_{\mathcal{B}_{B_0(u)}} \circ f_{e} \circ \mathcal{T}_{e}^{-1}|_{F_e}= \mathcal{T}_{u}|_{\mathcal{B}_{B_0(u)}} \circ T_{e} \circ \mathcal{T}_{e}^{-1}|_{F_e},$$ we deduce that $f|_{F_e}$ is real affine function.\smallskip
\end{proof}

We recall now some terminology taken from \cite{BuChu}. Let $K(H,H')$ be the space of all compact linear operators between two complex Hilbert spaces. We shall write $K(H)$ instead of $K(H,H)$. If $C_j$ is a Cartan factor of type $j\in \{1,2,3,4,5,6\}$, we define $K_1 = K(H,H')$ for $C_1= L(H,H')$, $K_j=C_j\cap K(H)$ for $j=2,3$, and in the remaining cases $K_4 = C_4,$ $K_5 = C_5$, and $K_6 = C_6$. The JB$^*$-triples $K_1,K_2,\ldots, K_6$ are called elementary JB$^*$-triples.  Suppose $E= \bigoplus_i^{\infty} C_i$ is an atomic JBW$^*$-triple, where each $C_i$ is a Cartan factor. It is known that the $c_0$-sum $K(E) = \bigoplus_i^{c_0} K_i$ is a weakly compact JB$^*$-triple and a triple ideal of $E$ with $K(E)^{**} = E$ (see \cite[Remark 2.6]{BuChu}). JB$^*$-triples of the form $K(E)$ are called weakly compact JB$^*$-triples (see \cite{BuChu}). It is further known that every element $x$ in a weakly compact JB$^*$-triple $K(E)$ can be written as a norm convergent (possibly finite) sum $\displaystyle x = \sum_{n=1} \lambda_n e_n$, where $e_n$ are mutually orthogonal minimal tripotents in $K(E)$ (and in $E$), and $(\lambda_n)\subseteq \mathbb{R}_0^+$ with $(\lambda_n)\to 0$ (see \cite[Remark 4.6]{BuChu}).\smallskip

Now, just a final technical step is separating us from our main goal.

\begin{proposition}\label{p properties for finite linear combinations of minimal tripotents} Let $E$ and $B$ be atomic JBW$^*$-triples, and suppose that $f: S(E) \to S(B)$ is a surjective isometry. Then the following statements hold:\begin{enumerate}[$(a)$]\item  For each minimal tripotent $e\in E$ we have $T_e(v) = f(v)$ for every minimal tripotent $v\in E_0(e)$, where $T_e : E_0(e) \to B_0(f(e))$ is the surjective real linear isometry given by Theorem \ref{t surjective isometries between the spheres preserve minimal tripotents};
\item Let $e_1,\ldots,e_n$ be mutually orthogonal minimal tripotents in $E$, and let $\lambda_1,\ldots,\lambda_n$ be positive real numbers with $\max_j \lambda_j=1$. Then $$f\left(\sum_{j=1}^n \lambda_j e_j\right) = \sum_{j=1}^n \lambda_j f\left(e_j\right);$$
\item $f$ maps $S(K(E))$ onto $S(K(B))$;
\item For each minimal tripotent $e\in E$ we have $f(u) =  T_e (u)$ for every non-zero tripotent $u\in E_0(e)$;
\item Let $v_1,\ldots,v_n$ be mutually orthogonal non-zero tripotents in $E$, and let $\lambda_1,\ldots,\lambda_n$ be positive real numbers with $\max_j \lambda_j=1$. Then $$f\left(\sum_{j=1}^n \lambda_j v_j\right) = \sum_{j=1}^n \lambda_j f\left(v_j\right).$$
\end{enumerate}
\end{proposition}

\begin{proof}$(a)$ Let $v$ be a minimal tripotent in $E_0(e)$. By Theorem \ref{t surjective isometries between the spheres preserve minimal tripotents} $f(e)$, $f(-e)$, $f(v)$ and $f(-v)$ are minimal tripotents in $B$ with $\|f(e)-f(-e)\| = \| 2e \| = 2$. Lemma \ref{p l 3.4 new} implies that $f(-e) = - f(e)$. Similarly,  $f(-v) = - f(v)$. By hypothesis $$\|f(e) \pm f(v) \| = \| f(e) - f(\mp v)\| = \|e\pm v\| =1.$$ By \cite[$(6)$ in page 360]{FerMarPe2012} we have $$\hbox{cp} (\{f(e)\}) = \{ y\in B : \|y\pm f(e)\|\leq 1 \} =\mathcal{B}_{B_0(e)}.$$ In particular $f(v)\in \mathcal{B}_{B_0(e)}.$\smallskip

The mapping $T_e$ is a surjective real linear isometry between JBW$^*$-triples. It follows from \cite[Theorem 4.8]{IsKaRo95} that $T_e$ preserves the symmetrized triple product $\langle x,y,z \rangle := \frac13 (\{x,y,z\}+\{z,x,y\}+\{y,z,x\})$. In particular $T_e$ preserves cubes of the form $x^{[3]}= \{x,x,x\}$ and maps minimal tripotents to minimal tripotents. Therefore $T_e(v)$ is a minimal tripotent in $B_0(e)$.\smallskip

Again by Theorem \ref{t surjective isometries between the spheres preserve minimal tripotents} we have $$f(v) + T_v(e) =f(e+v) = f(e) + T_e (v).$$
Since $f(v)\perp f(e)$, we deduce that $$f(v) = P_2(f(v)) (f(v) + T_v(e)) = P_2(f(v)) (f(e) + T_e (v)) = P_2(f(v)) (T_e (v)).$$ Since $\|T_e (v)\|= \|v\|=1$, Lemma 1.6 in \cite{FriRu85} implies that $T_e(v) = f(v) + P_0(f(v)) T_e( v)$. Finally, since $T_e(v)$ is a minimal tripotent we get $P_0(f(v)) T_e( v)=0$, $T_e(v) = f(v)$.\smallskip

$(b)$ Let $e_1,\ldots,e_n$ be mutually orthogonal minimal tripotents in $E$, and let $\lambda_1,\ldots,\lambda_n$ be positive real numbers with $\max_j \lambda_j=1$. We may assume $\lambda_1=1$. Let $T_{e_1}$ be the surjective real linear isometry given by Theorem \ref{t surjective isometries between the spheres preserve minimal tripotents}. We deduce from the just quoted Theorem and the statement in $(a)$ that $$f\left(\sum_{j=1}^n \lambda_j e_j\right)= f(e_1) + T_{e_1} \left( \sum_{j=2}^n \lambda_j e_j \right)= f(e_1) + \sum_{j=2}^n \lambda_j  T_{e_1} \left(e_j \right) = \sum_{j=1}^n \lambda_j f\left(e_j\right).$$

$(c)$ We have already commented that every element in $K(E)$ can be approximated in norm by an element of the form $\sum_{j=1}^n \lambda_j e_j$, where $e_1,\ldots,e_n$ are mutually orthogonal minimal tripotents in $E$, and $\lambda_1,\ldots,\lambda_n$ are positive real numbers. Consequently, elements in $S(K(E))$ can be approximated in norm by finite sums of the form $\sum_{j=1}^n \lambda_j e_j$, with $e_1,\ldots,e_n$ and $\lambda_1,\ldots,\lambda_n$ as above and $\max_j \lambda_j=1$. If we observe that, by Theorem \ref{t surjective isometries between the spheres preserve minimal tripotents} and $(b)$, $\displaystyle f\left(\sum_{j=1}^n \lambda_j e_j\right)= \sum_{j=1}^n \lambda_j f\left(e_j\right)\in S(K(B))$, we conclude that $f(S(K(E)))\subseteq S(K(B))$. Applying the same argument to $f^{-1}$ we get $f(S(K(E)))= S(K(B))$.\smallskip

$(d)$ Let $u$ be a non-zero tripotent in $E_0(e)$, where $e$ is a minimal tripotent in $E$. We may assume that $u$ is not minimal, otherwise the desired statement follows from $(a)$. Thus, since $E$ is atomic, we can find a minimal tripotent $v$ in $E_0(e)$ with $u \geq v$. By Theorem \ref{t surjective isometries between the spheres preserve minimal tripotents} we have $$f(e) + T_e (u) = f(e +u ) = f(e + v+ (u-v)) = f(v) + T_v (e) + T_v (u-v)$$ $$ = \hbox{(by $(a)$)} = f(v) + f(e) + T_v(u-v) = f(e) + f(v + u-v) = f(e) + f(u),$$ which proves the desired statement.\smallskip

$(e)$ Under the assumptions there exists $k\in \{1,\ldots,n\}$ with $\lambda_k=1$. Since $E$ is atomic, we can find a minimal tripotent $e_k\leq v_k$. Let $T_{e_k}$ be the surjective real linear isometry given by Theorem \ref{t surjective isometries between the spheres preserve minimal tripotents}. We deduce from the just quoted Theorem and the statement in $(d)$ that $$f\left(\sum_{j=1}^n \lambda_j v_j\right)= f(e_k) + T_{e_k} \left( (v_k-e_k)+ \sum_{j\neq k} \lambda_j v_j \right)$$ $$= f(e_k) + T_{e_k} \left( (v_k-e_k) \right) + \sum_{j\neq k} \lambda_j T_{e_k} \left(v_j \right)$$ $$= f(e_k + v_k -e_k) + \sum_{j\neq k} \lambda_j  f \left(v_j \right)= \sum_{j=1}^n \lambda_j f\left(v_j\right).$$
\end{proof}

We can establish now our main result.

\begin{theorem}\label{t Tingley on atomic JBW-triples} Let $f: S(E) \to S(B)$ be a surjective isometry, where $E$ and $B$ are atomic JBW$^*$-triples. Then there exists a {\rm(}unique{\rm)} real linear isometry $T:E\to B$ such that $f=T_{|S(E)}$.
\end{theorem}

\begin{proof} Let $K(E)$ and $K(B)$ denote the ideals of $E$ and $B$ generated by the minimal tripotents in $E$ and $B$, respectively. We deduce from Proposition \ref{p properties for finite linear combinations of minimal tripotents}$(c)$ that $f (S(K(E))) = S(K(B))$ and $f|_{S(K(E))} : S(K(E))\to S(K(B))$ is a surjective isometry. We observe that $K(E)$ and $K(B)$ are weakly compact JB$^*$-triples in the sense of \cite{BuChu,FerPe17}, so by \cite[Theorem 2.5]{FerPe17} there exists a surjective real linear isometry $S : S(K(E))\to S(K(B))$ satisfying $f(x) = S(x)$, for every $x\in S(K(E))$.\smallskip

The mapping $T= S^{**} : K(E)^{**} =E\to K(B)^{**}=B$ is a surjective real linear isometry and a weak$^*$ continuous mapping. By construction $T(x) = S(x) = f(x)$, for every $x\in S(K(E))$.\smallskip

We claim that \begin{equation}\label{eq coincidence on tripotents} T(w) = f(w), \hbox{ for every non-zero tripotent } w\in E.
\end{equation}  Since $E$ is atomic, we can find a minimal tripotent $e\leq w$. The mapping $T_e : E_0(e) \to B_0 (f(e))$ given by Theorem \ref{t surjective isometries between the spheres preserve minimal tripotents} is a surjective real linear isometry between (atomic) JBW$^*$-triples. By \cite[Proposition 2.3$(1.)$]{MarPe} $T_e$ also is weak$^*$-continuous. By construction $T_e (v) = f(v) = S(v) = T(v)$ for every finite-rank tripotent $v\in E_0(e)$. Since every tripotent in an atomic JBW$^*$-triple can be approximated in the weak$^*$-topology by a net of finite-rank tripotents, we deduce from the above that $T(w) = T_e (w)$ for every tripotent in $E_0(e)$. Now, by Theorem \ref{t surjective isometries between the spheres preserve minimal tripotents} and Proposition \ref{p properties for finite linear combinations of minimal tripotents}$(d)$ or $(e)$, we have $$f(w) = f(e) + T_e (w-e) = f(e) + T(w-e) = T(e) + T(w-e) = T(w),$$ which proves the claim.\smallskip

It is known that in a JBW$^*$-triple $M$ the set of tripotents in $M$ is norm-total, that is, every element in $M$ can be approximated in norm by finite real linear combinations of mutually orthogonal non-zero tripotents in $M$ (see \cite[Lemma 3.11]{Horn87}. However, a general JBW$^*$-triple need not contain a single minimal tripotent). Combining this fact with \eqref{eq coincidence on tripotents} and Proposition \ref{p properties for finite linear combinations of minimal tripotents}$(e)$, we can easily conclude that $T(x) = f(x),$ for every $x\in S(E)$.
\end{proof}

\textbf{Acknowledgements} Authors partially supported by the Spanish Ministry of Economy and Competitiveness and European Regional Development Fund project no. MTM2014-58984-P and Junta de Andaluc\'{\i}a grant FQM375.


\begin{thebibliography}{0}

\bibitem{AkPed92} C.A. Akemann, G.K. Pedersen, Facial structure in operator algebra theory, \emph{Proc. Lond. Math. Soc.} \textbf{64}, 418-448 (1992).



\bibitem{BarTi86} T. Barton abd R. M. Timoney, Weak$^*$-continuity of Jordan triple products and applications, \emph{Math. Scand.} \textbf{59}, 177-191 (1986).




\bibitem{BuChu} L.J. Bunce and C.-H. Chu,  Compact  operations, multipliers  and Radon-Nikodym property in $JB^*$-triples, {\it Pacific J. Math.} \textbf{153}, 249-265 (1992).

\bibitem{BuFerMarPe} L.J. Bunce, F.J. Fern\'andez-Polo, J. Mart{\'i}nez Moreno, A.M. Peralta, A Sait\^o-Tomita-Lusin theorem for JB$^*$-triples and applications, Quart. J. Math. Oxford, {\bf 57}, 37-48 (2006).

\bibitem{BurFerGarMarPe} M. Burgos, F.J. Fern{\'a}ndez-Polo, J. Garc{\'e}s, J. Mart{\'\i}nez, A.M. Peralta, Orthogonality preservers in C$^*$-algebras, JB$^*$-algebras and JB$^*$-triples, \emph{J. Math. Anal. Appl.} \textbf{348}, 220-233 (2008).



\bibitem{ChengDong} L. Cheng, Y. Dong, On a generalized Mazur-Ulam question: extension of isometries between unit spheres of Banach spaces,
\emph{J. Math. Anal. Appl}. \textbf{377} 464-470 (2011).

\bibitem{Chu2012} Ch.-H. {Chu}. {\em {Jordan structures in geometry and analysis.}}, Cambridge, Cambridge University Press, 2012.






\bibitem{Di86} S. Dineen, The second dual of a JB$^*$-triple system, In: Complex analysis, functional analysis and approximation theory (ed. by J. M\'ugica), 67-69, (North-Holland Math. Stud. 125), North-Holland, Amsterdam-New York, (1986).









\bibitem{EdFerHosPe2010} C.M. Edwards, F.J. Fern{\'a}ndez-Polo, C.S. Hoskin, A.M. Peralta, On the facial structure of the unit ball in a JB$^*$-triple, \emph{J. Reine Angew. Math.} \textbf{641}, 123-144 (2010).


\bibitem{EdRutt88} C.M. Edwards, G.T. R\"{u}ttimann, On the facial structure of the unit balls in a JBW$^*$-triple and its predual, \emph{J. Lond. Math. Soc.} \textbf{38}, 317-332 (1988).

\bibitem{EdRu96} C.M. Edwards, G.T. R\"uttimann, Compact tripotents in bi-dual JB$^*$-triples, \emph{Math. Proc. Camb. Phil. Soc.} \textbf{120}, 155-173 (1996).



\bibitem{FerMarPe2012} F.J. Fern{\'a}ndez-Polo, J. Mart{\'i}nez, A.M. Peralta, Contractive perturbations in JB$^*$-triples, \emph{J. London Math. Soc.} (2) \textbf{85} 349-364  (2012).

\bibitem{FerPe06} F.J. Fern\'andez-Polo, A.M. Peralta, Closed tripotents and weak compactness in the dual space of a JB$^*$-triple, \emph{J. London Math. Soc.} \textbf{74}, 75-92 (2006).




\bibitem{FerPe17} F.J. Fern\'andez-Polo, A.M. Peralta, Low rank compact operators and Tingley's problem, preprint 2016. arXiv:1611.10218v1

\bibitem{FerPe17b} F.J. Fern\'andez-Polo, A.M. Peralta, On the extension of isometries between the unit spheres of a C$^*$-algebra and $B(H)$, preprint 2017. arXiv:1701.02916v1

\bibitem{FriRu85} Y.~{Friedman} and B.~{Russo}.
\newblock {Structure of the predual of a $JBW\sp*$-triple.}
\newblock {\em {J. Reine Angew. Math.}}, \textbf{356}, 67-89 (1985).

\bibitem{FriRu86} Y.~{Friedman} and B.~{Russo}, The Gel'fand-Naimark theorem for JB$^*$-triples, \emph{Duke Math. J.} \textbf{53}, no. 1, 139-148 (1986).




\bibitem{Hein80} S. {Heinrich}, {Ultraproducts in Banach space theory,} \emph{J. Reine Angew. Math.} \textbf{313}, 72-104 (1980).

\bibitem{Horn87} G. Horn, Characterization of the predual and ideal structure of a JBW$^*$-triple, \emph{Math. Scand.} \textbf{61}, no. 1, 117-133 (1987).

\bibitem{IsKaRo95} J.M. Isidro, W. Kaup, A. Rodr{\'\i}guez, On real forms of JB$^*$-triples, \emph{Manuscripta Math.} \textbf{86}, 311-335 (1995).




\bibitem{Ka83} W. Kaup, A Riemann Mapping Theorem for bounded symmentric domains in complex Banach spaces, \emph{Math. Z.} \textbf{183}, 503-529 (1983).

\bibitem{Ka96} W. Kaup, On spectral and singular values in JB$^*$-triples, \emph{Proc. Roy. Irish Acad. Sect. A} \textbf{96}, no. 1, 95-103 (1996).




\bibitem{Mank1972} P. Mankiewicz, On extension of isometries in normed linear spaces, \emph{Bull. Acad. Pol. Sci., S\'{e}r. Sci. Math. Astron. Phys.} \textbf{20} 367-371 (1972).

\bibitem{MarPe} J. Mart\'{\i}nez and A. M. Peralta, Separate weak*-continuity
of the triple product in dual real JB*-triples, \emph{Math. Z.}
\textbf{234}, 635-646 (2000).


\bibitem{NeRu} M. Neal, B. Russo, Operator space characterizations of $C\sp *$-algebras
and ternary rings, \emph{Pacific J. Math.} {\bf 209}, no. 2, 339-364 (2003).






\bibitem{PeTan16} A.M. Peralta, R. Tanaka, A solution to Tingley's problem for isometries between the unit spheres of compact C$^*$-algebras and JB$^*$-triples, preprint 2016. arXiv:1608.06327v1.










\bibitem{Tan2014} R. Tanaka, A further property of spherical isometries, \emph{Bull. Aust. Math. Soc.}, \textbf{90}, 304-310 (2014).

\bibitem{Tan2016} R. Tanaka, The solution of Tingley's problem for the operator norm unit sphere of complex $n \times n$ matrices, \emph{Linear Algebra Appl.} \textbf{494}, 274-285 (2016).

\bibitem{Tan2016-2}
R. Tanaka, Spherical isometries of finite dimensional $C^*$-algebras, to appear in \emph{J. Math. Anal. Appl.}

\bibitem{Tan2016preprint} R. Tanaka, Tingley's problem on finite von Neumann algebras, preprint 2016.

\bibitem{Ting1987} D. Tingley, Isometries of the unit sphere, \emph{Geom. Dedicata} \textbf{22}, 371-378 (1987).


\bibitem{YangZhao2014} X. Yang, X. Zhao, On the extension problems of isometric and nonexpansive mappings. In: \emph{Mathematics without boundaries}. Edited by Themistocles M. Rassias and Panos M. Pardalos. 725-748, Springer, New York, 2014.


\end{thebibliography}
\end{document}